\documentclass[11pt]{amsart}
\usepackage[colorlinks=true,pagebackref,
  hyperindex,citecolor=green,linkcolor=blue]{hyperref}
\usepackage{tikz-cd} \usepackage{amsmath,amssymb,amsthm}
\usepackage{comment}
\usepackage{xcolor}
\usepackage[textwidth=3.3 cm,textsize=small,shadow
]{todonotes}

\newtheorem{theorem}{Theorem}[section]

\newtheorem{corollary}[theorem]{Corollary}
\newtheorem{proposition}[theorem]{Proposition}

\theoremstyle{definition}
\newtheorem{definition}[theorem]{Definition}

\newtheorem{setting}[theorem]{Setting}
\newtheorem{notation}[theorem]{Notation}

\newtheorem{remark}[theorem]{Remark}
\newtheorem{example}[theorem]{Example}

\theoremstyle{remark}

\numberwithin{equation}{section}

\input xy
\xyoption{all}

\newcommand{\height}{\mbox{height}\,}
\newcommand{\mfm}{\mbox{$\mathfrak{m}$}}
\newcommand{\mfa}{\mbox{$\mathfrak{a}$}}
\newcommand{\mfb}{\mbox{$\mathfrak{b}$}}
\newcommand{\mfc}{\mbox{$\mathfrak{c}$}}

\newcommand{\cO}{{\mathcal O}}
\newcommand{\lra}{{\longrightarrow}}

\newcommand{\bC}{\mathbb{C}}

\newcommand{\bZ}{{\mathbb{Z}}}
\newcommand{\bQ}{{\mathbb{Q}}}
\newcommand{\rad}{\mbox{\rm{rad}}\,}
\newcommand{\Min}{\mbox{\rm{Min}}}
\newcommand{\rt}{\mathrm{rt}}
\newcommand{\rn}{\mathrm{rn}}

\newcommand{\fzero}{\mbox{$\underline{\rm f_0t}$}}
\newcommand{\ftwo}{\mbox{$\underline{\rm f_2t}$}}
\newcommand{\fim}{\mbox{$\underline{\rm f_{i-1}t}$}}
\newcommand{\fit}{\mbox{$\underline{\rm f_it}$}}
\newcommand{\fip}{\mbox{$\underline{\rm f_{i+1}t}$}}
\newcommand{\fipp}{\mbox{$\underline{\rm f_{i+2}t}$}}

\newcommand{\fn}{\mbox{$\underline{\rm f_{n}t}$}}
\newcommand{\fnp}{\mbox{$\underline{\rm f_{n+1}t}$}}
\newcommand{\ft}{\mbox{$\underline{\rm ft}$}}
\newcommand{\f}{\mbox{$\underline{\rm f}$}}

\begin{document}

\title[Divisors of expected Jacobian type]{Divisors of expected
  Jacobian type}

\author[J. \`Alvarez Montaner]{Josep \`Alvarez Montaner}

\author[F. Planas-Vilanova]{Francesc Planas-Vilanova}

\address{Departament de Matem\`atiques, Universitat Polit\`ecnica de
  Catalunya. Diagonal 647, Barcelona} \email{Josep.Alvarez@upc.edu,
  Francesc.Planas@upc.edu}

\thanks{ Both authors are supported by the Spanish Ministerio de
  Econom\'ia y Competitividad MTM2015-69135-P and Generalitat de
  Catalunya SGR2017-932. }

\date{\today}

%\keywords {Betti numbers, Monomial ideals}

%\subjclass[2000]{Primary 13D45, 13N10}

\begin{abstract}
Divisors whose Jacobian ideal is of linear type have received a lot of
attention recently because of its connections with the theory of
$D$-modules.  In this work we are interested on divisors of expected
Jacobian type, that is, divisors whose gradient ideal is of linear
type and the relation type of its Jacobian ideal coincides with the
reduction number with respect to the gradient ideal plus one. We
provide conditions in order to be able to describe precisely the
equations of the Rees algebra of the Jacobian ideal. We also relate
the relation type of the Jacobian ideal to some $D$-module theoretic
invariant given by the degree of the Kashiwara operator.
\end{abstract}

\maketitle

\section{Introduction}\label{Sec1}
Let $(X,O)$ be a germ of a smooth $n$-dimensional complex variety and
$\cO_{X,O}$ the ring of germs of holomorphic functions in a
neighbourhood of $O$, which we identify with $R=\bC\{x_1,\dots ,
x_n\}$ by taking local coordinates.  Let $D_R[s]$ be the polynomial
ring in an indeterminate $s$ with coefficients in the ring of
differential operators $D_R= R\langle \partial_1, \dots , \partial_n
\rangle$ where $\partial_i$ are the partial derivatives with respect
to the variables $x_i$.  To any hypersurface defined by $f\in R$ we
may attach several invariants coming from the theory of $D$-modules
that measure its singularities.  The goal of this work is to get more
insight on the {\it parametric annihilator} ${\rm
  Ann}_{D_R[s]}(\boldsymbol{f^s}):= \{ P(s) \in D_R[s] \hskip 2mm
| \hskip 2mm P(s) \cdot \boldsymbol{f^s}=0 \},$ where we understand
$\boldsymbol{f^s}$ as a formal symbol that takes the obvious meaning
$f^r$ when specializing to any integer $r\in \bZ$.  This is the
defining ideal of the $D_R[s]$-modules generated by
$\boldsymbol{f^s}$, that we denote as $D_R[s] \boldsymbol{f^s}$, which
plays a key role in the theory of {\it Bernstein-Sato polynomials} as
shown by Kashiwara \cite{Kas77}. Among the differential operators
annihilating $\boldsymbol{f^s}$ there exists the so-called {\it
  Kashiwara operator} \cite[Theorem 6.3]{Kas77} that has been used in
some of the first algorithmic approaches to the computation of
Bernstein-Sato polynomials given by Yano \cite{Yan78} and Brian\c{c}on
et al. \cite{BGMM}. Furthermore, the degree of the Kashiwara operator
is an interesting analytic invariant of the singularity, although it
is much coarser than the Bernstein-Sato polynomial itself.

\vskip 2mm

A common theme in the study of the parametric annihilator is whether
it is generated by operators of degree one.  This and some related
linearity properties have been used by several authors in a wide range
of different problems \cite{CN02}, \cite{Tor07}, \cite{Nar08},
\cite{CN09} , \cite{Arc10}, \cite{Nar15}, \cite{Wal15}, \cite{Wal17}.
This linearity property of differential operators can be checked using
algebraic methods as it was proved by Calder\'on-Moreno and
Narv\'ez-Macarro in \cite{CN02}. Namely, this property holds whenever
the Jacobian ideal of the hypersurface $f$ is of linear type, that is
the Rees algebra and the symmetric algebra of the Jacobian ideal
coincide.

\vskip 2mm

The aim of this paper is to get further connections between the Rees
algebra of the Jacobian ideal and the parametric annihilator.
Building upon work of Mui\~nos and the second author in \cite{MP12},
we introduce in Definition~\ref{defAlmost} the notion of divisors of
{\it expected Jacobian type} as those divisors whose gradient ideal is
of linear type and whose Jacobian ideal has relation type equal to its
reduction number plus one (see also Definition~\ref{defexp}). In
Remark~\ref{remJacImplies}, it is easily seen that this is a natural
generalization of the divisors of {\it linear Jacobian type}
considered in \cite{CN02} (see also \cite{Nar08}). For divisors of
linear Jacobian type we can find an equation that resembles the
initial term or symbol of the Kashiwara operator with respect to a
given order, and indeed this is the case under some extra conditions.

\vskip 2mm

The organization of the paper is as follows: in
Section~\ref{SecEquations} we review the basics on the equations of
Rees algebras and recover and extend some of the results of Mui\~nos
and the second author in \cite{MP12}.  Section \ref{SecDefExp} and
~\ref{SecDefAlm} are devoted to introduce the notion of ideal of expected relation type and its specialization to the case
of the Jacobian ideal of a hypersurface.
In Section~\ref{SecConnections} we describe the connection between
divisors of expected Jacobian type and the parametric annihilator.  We
relate the degree of the Kashiwara operator with the relation type of
the Jacobian ideal in Proposition \ref{Lf}.  In
Section~\ref{SecExamples} we present several examples in which we
explore the case in which the ideal has the expected relation type. We
also study some cases in which this condition is not satisfied.

\vskip 2mm

Any unexplained notation or definition can be found in \cite{BH93} or
\cite{SH}. Throghout the paper, $(R,\mfm)$ is a Noetherian local ring
and $\mfb\subseteq\mfa$ and $J\subseteq I$ are ideals of $R$.

\vskip 5mm

{ {\bf Acknowledgements:} This work grew up from early conversations with Ferran Mui\~nos and we are really grateful for his insight.
We would also thank Jos\'e Mar\'ia Giral for some helpful comments. Part of this work was done during a research stay of the first author at CIMAT, Guanajuato with a Salvador de Maradiaga grant (ref. PRX 19/00405) from the Ministerio de Ciencia, Innovaci\'on y Universidades.}

\section{On the equations of Rees algebras}\label{SecEquations}

Let $(R,\mfm)$ be a Noetherian local ring and let $\mfa=(f_1,\dots,
f_{m})$ be an ideal of $R$, $m\geq 1$. Let ${\bf R}(\mfa)=R[\mfa
  t]=\bigoplus_{d\geq 0} \mfa^dt^d \subset R[t]$ be the {\em Rees
  algebra of $\mfa$}. Let $A=R[\xi_1,\ldots ,\xi_{m}]$ be a polynomial
ring in a set of variables $\xi_1,\dots ,\xi_{m}$ and coefficients in
$R$. Consider the graded surjective morphism $\varphi:A\to {\bf
  R}(\mfa)$ sending $\xi_i$ to $f_i t$, for $i=1,\dots,n+1$. The
kernel of this morphism is a graded ideal $Q=\bigoplus_{d\geq 1} Q_d$,
whose elements will be referred to as the {\em equations of ${\bf
    R}(\mfa)$}. Let $Q\langle d \rangle$ be the ideal generated by the
homogeneous equations of degree at most $d$. We then have an
increasing sequence $Q\langle 1\rangle\subseteq Q\langle 2\rangle
\subseteq \cdots \subseteq Q$ that stabilizes at some point. The
smallest integer $L\geq 1$ such that $Q\langle L \rangle = Q$ is the
{\em relation type of ${\bf R}(\mfa)$} and will be denoted
$\rt(\mfa)$. We say that $\mfa$ is an {\em ideal of linear type} when
$\rt(\mfa)=1$.

\vskip 2mm

Observe that the ideal $Q$ depends on the polynomial presentation
$\varphi$. Nevertheless, the quotients $(Q/Q\langle d-1\rangle)_d$,
for $d\geq 2$, do not (see \cite{Pla98}). Indeed, let $\alpha:{\bf
  S}(\mfa)\to {\bf R}(\mfa)$ be the canonical graded surjective
morphism between the symmetric algebra ${\bf S}(\mfa)$ of $\mfa$ and
the Rees algebra ${\bf R}(\mfa)$ of $\mfa$. Given $d\geq 2$, the
$d$-th module of effective relations of $\mfa$ is defined to be
$E(\mfa)_d=\ker(\alpha_d)/\mfa\cdot\ker(\alpha_{d-1})$. One shows
that, for $d\geq 2$, $E(\mfa)_d\cong (Q/Q\langle d-1\rangle)_d$. In
particular, the relation type of $\mfa$ can be calculated as the least
integer $L\geq 1$, such that $E(\mfa)_d=0$, for all $d\geq
L+1$. Moreover, it is known that $E(\mfa)_d\cong
H_1(f_1t,\ldots,f_{m}t;{\bf R}(\mfa))_d$, where the right-hand module
stands for the degree $d$-component of the first Koszul homology
module associated to the sequence of degree one elements
$f_1t,\ldots,f_{m}t$ of ${\bf R}(\mfa)$ (\cite[Theorem~2.4]{Pla98}).

\vskip 2mm
 
The characterization of $E(\mfa)_d$ in terms of the Koszul homology
was used in \cite{MP12} in order to obtain the equations of ${\bf
  R}(\mfa)$ for equimultiple ideals $\mfa$ of deviation one. Our
purpose in this section is to rephrase, and extend a little bit, some
of those results, but doing more emphasis in the Koszul conditions
than in the ``regular sequence type conditions''. These
characterizations will be applied in the next sections to the 
Jacobian ideal of a hypersurface. For the sake of completeness and
self-containment, we outline parts of the line of reasoning in
\cite{MP12}. Let us start by setting our general notations.

\begin{setting}\label{setting}
Let $(R,\mfm)$ be a Noetherian local ring, $n\geq 2$. Let
$f_1,\ldots,f_n\in \mfm$ and $f=f_{n+1}\in \mfm$. Let $J=(f_1,\ldots
,f_n)$ and $I=(f_1,\ldots,f_n,f)=(J,f)$ be ideals of $R$. For
$i=1,\ldots ,n+1$, let $J_i=(f_1,\ldots,f_i)$; set $J_0=0$ and observe
that $J_n=J$ and $J_{n+1}=I$.
\begin{eqnarray*}
  &&\mbox{For }i=1,\ldots,n+1\mbox{ and }d\geq 2,\mbox{ set
  }\phantom{++}T_{i,d}=\frac{(J_{i-1}I^{d-1}:f_i)\cap
    I^{d-1}}{J_{i-1}I^{d-2}}.\\ &&\mbox{For }d=1,\mbox{ set}
  \phantom{++}T_{i,1}=J_{i-1}:f_i.
\end{eqnarray*}
Note that for $i=1$ (and any $d\geq 1$), then $T_{1,d}=(0:f_1)\cap
I^{d-1}$. For $i=n+1$ and $d\geq 2$, it was shown in \cite[Proof of
  Lemma~3.1]{MP12} that:
\begin{eqnarray}\label{tn1}
  T_{n+1,d}=\frac{(JI^{d-1}:f)\cap I^{d-1}}{JI^{d-2}}\cong
  \frac{JI^{d-1}:f^d}{JI^{d-1}:f^{d-1}}.
\end{eqnarray}
The isomorphism goes as follows. Given $a\in (JI^{d-1}:f)\cap
I^{d-1}$, since $I^{d-1}=JI^{d-1}+f^{d-1}R$, write $a=b+cf^{d-1}$ with
$b\in JI^{d-1}$ and $c\in R$. The class of an element $a\in
(JI^{d-1}:f)\cap I^{d-1}$ is sent to the class of $c\in JI^{d-1}:f^d$.
\end{setting}

\begin{notation}\label{notKos}{\sc A graded Koszul complex.} 
Let us denote $K(z_1,\ldots,z_r;U)$ the Koszul complex of a sequence
of elements $z_1,\ldots ,z_r$ of a ring $U$. Since $U$ will always be
the Rees algebra ${\bf R}(I)$ of $I$, we just skip the letter $U$. For
$i=1,\ldots,n+1$, we consider the sequences $\fit:=f_1t,\ldots,f_{i}t$
of elements of degree one in ${\bf R}(I)$; we highlight the distinct
notation with the length one sequence $f_it$. Set
$\ft:=\fnp=f_1t,\ldots,f_nt,ft$. Thus $K(\fit)=K(f_1t,\ldots,f_it;{\bf
  R}(I))$ stands for the Koszul complex associated to
$\fit=f_1t,\ldots,f_it$, with first nonzero zero terms:
\begin{eqnarray*}
  K(\fit): \phantom{++} \ldots \to
  K_2(\fit)\stackrel{\partial_{2}}{\longrightarrow}
  K_1(\fit)\stackrel{\partial_{1}}{\longrightarrow} K_0(\fit)\to 0.
\end{eqnarray*}
Let $H_j(\fit)=H_j(K(\fit))$ be its $j$-th homology module. Note that,
since ${\bf R}(I)$ is a graded algebra, $K(\fit)$, and hence its
homology, inherit a natural grading. The first nonzero terms of the
degree $d$-component $K(\fit)_d$, $d\geq 2$, (omitting the powers of
the variable $t$) are:
\begin{eqnarray*}
\ldots \to K_2(\fit)_d=\wedge_{2}(R^{i})\otimes
I^{d-2}\stackrel{\partial_{2,d-2}}{\longrightarrow}
K_1(\fit)_d=\wedge_{1}(R^{i})\otimes
I^{d-1}\stackrel{\partial_{1,d-1}}{\longrightarrow}K_0(\fit)_d=I^d\to
0.
\end{eqnarray*}
The Koszul differentials are defined as follows: if $e_1,\ldots ,e_i$
stands for the canonical basis of $R^i$ and $u\in I^{d-2}$ and $v\in
I^{d-1}$, then
\begin{eqnarray*}
\partial_{2,d-2}(e_j\wedge e_l\otimes u)=e_l\otimes f_ju-e_j\otimes
f_lu \; \mbox{ and } \; \partial_{1,d-1}(e_j\otimes v)=f_jv.
\end{eqnarray*}
Note that, under the isomorphism $\wedge_{1}(R^{i})\otimes
I^{d-1}\cong I^{d-1}\oplus\stackrel{(i)}{\cdots} \oplus I^{d-1}$, the
differential $\partial_{1,d-1}$ sends the $i$-th tuple
$(a_1,\ldots,a_i)\in (I^{d-1})^{\oplus i}$ to the element
$a_1f_1+\cdots+a_if_i\in I^{d}$. In particular, for $d=1$,
$H_1(\fit)_1=\{(a_1,\ldots,a_i)\in R^i\mid
\sum_j^ia_jf_j=0\}=Z_1(f_1,\ldots,f_i)$, the first module of syzygies
of $J_i=(f_1,\ldots,f_i)$.
\end{notation}

\begin{remark}\label{remexpl}{\sc Equations vs cycles.}
Let $Q$ be the ideal of equations of ${\bf R}(I)$. As said before,
\begin{eqnarray}\label{effectivekos}
  E(I)_d\cong \left(\frac{Q}{Q\langle d-1\rangle}\right)_d\cong
  H_1(\ft)_d=H_1(K(\ft))_d=H_1(f_1t,\ldots,f_nt,ft;{\bf R(I)})_d,
\end{eqnarray}
i.e., the $d$-th module of effective relations $E(I)_d$ of $I$ is
isomorphic to the degree $d$-component of the first Koszul homology
module $H_1(\ft)$ of $\ft$, where $d\geq 2$. This isomorphism sends
the class of an equation $P\in Q_d$ to the class of the cycle
$(P_1(\f),\ldots,P_n(\f),P_{n+1}(\f))\in
\bigoplus_{j=1}^{n+1}I^{d-1}$, where $\f=f_1,\ldots,f_n,f$,
$P=\sum_{j=1}^{n+1}\xi_jP_{j}$, and $P_j\in
A_{d-1}=R[\xi_1,\ldots,\xi_{n+1}]_{d-1}$. (See
\cite[Remark~2.1]{MP12}.)
\end{remark}

The next two remarks are devoted to write more explicitly some
complexes and morphisms that will be used subsequently.

\begin{remark}\label{remshort}
{\sc A short exact sequence of Koszul complexes.} Let $K(f_it)$ be the
Koszul complex associated to the length one sequence $f_it\in {\bf
  R}(I)$. So $K_0(f_it)={\bf R}(I)$, $K_1(f_it)=\wedge_1(R)\otimes
{\bf R}(I)\cong {\bf R}(I)$, and $K_j(f_it)=0$, for $j\neq 0,1$. In
degree $d\geq 1$, $K_0(f_it)_d=I^d$, $K_1(f_it)=(\wedge_1(R)\otimes
{\bf R}(I))_d\cong I^{d-1}$ and, for $a\in I^{d-1}$, then
$\partial_{1,d-1}(a)=af_i$.

\vskip 2mm

There is an isomorphism of Koszul complexes $K(\fit)\cong
K(\fim)\otimes K(f_it)$. Concretely,
\begin{multline*}
K_p(\fit)\cong\bigoplus_{r+s=p}K_r(\fim)\otimes
K_s(f_it)=K_{p}(\fim)\otimes K_0(f_it) \oplus K_{p-1}(\fim)\otimes
K_1(f_it)\cong \\ K_{p}(\fim)\otimes {\bf R}(I)\oplus
K_{p-1}(\fim)\otimes {\bf R}(I)\cong K_{p}(\fim)\oplus K_{p-1}(\fim),
\end{multline*}
which induces a short exact sequence of Koszul complexes:
\begin{eqnarray}\label{eqShort}
0\to K(\fim)\to K(\fit)\to K(\fim)(-1)\to 0,
\end{eqnarray}  
where $K(\fim)(-1)$ is the shifted complex by -1, i.e.,
$K_s(\fim)(-1)=K_{s-1}(\fim)$. In particular, for $d\geq 1$, the
degree $d$-component gives rise to the the short exact sequence of
complexes:
\begin{eqnarray*}
0\to K(\fim)_d\to K(\fit)_d\to K(\fim)(-1)_d\to 0.
\end{eqnarray*}
Displaying by columns the first nonzero terms of each complex, we get:
\begin{equation*}
\xymatrix@C=2pc@R=2pc{
&\vdots\ar[d]&\vdots\ar[d]&\vdots\ar[d]&\\
  0\ar[r]&\wedge_2(R^{i-1})\otimes I^{d-2}\ar[r]\ar[d]&
  \wedge_2(R^{i})\otimes I^{d-2}\ar[r]\ar[d]&
  R^{i-1}\otimes I^{d-2}\ar[r]\ar[d]&0\\
  0\ar[r]&R^{i-1}\otimes I^{d-1}\ar[r]\ar[d]&
  R^{i}\otimes I^{d-1}\ar[r]\ar[d]&
  I^{d-1}\ar[r]\ar[d]&0\\
0\ar[r]&I^d\ar[r]\ar[d]&I^d\ar[r]\ar[d]&0\ar[r]\ar[d]&0\\
&0&0&0&.
}
\end{equation*}
The middle row, $0\to K_1(\fim)_d\to K_1(\fit)\to K_1(\fim)(-1)_d\to
0$, is nothing else than:
\begin{eqnarray*}
  0\to I^{d-1}\oplus\stackrel{(i-1)}{\cdots}\oplus I^{d-1}
  \longrightarrow I^{d-1}\oplus\stackrel{(i)}{\cdots}\oplus I^{d-1}
  \to I^{d-1}\to 0,
\end{eqnarray*}  
where the first morphism sends $(a_1,\ldots,a_{i-1})$ to
$(a_1,\ldots,a_{i-1},0)$, the inclusion, and the second morphism sends
$(a_1,\ldots,a_{i})$ to $a_i$, the projection to the last component.
\end{remark}

\begin{remark}\label{remlong}{\sc The long exact sequence in homology.}
In turn, the short exact sequence \eqref{eqShort} induces the long exact
sequence in homology. We display its degree $d$-component,
$d\geq 1$.
\begin{multline*}
  \ldots\to H_2(K(\fim)(-1))_d\stackrel{\delta}{\longrightarrow}
  H_1(\fim)_d\to H_1(\fit)_d\to
  H_1(K(\fim)(-1))_d\stackrel{\delta}{\longrightarrow}\\ \to
  H_0(\fim)_d\to H_0(\fit)\to H_0(K(\fim)(-1))_d\to 0.
\end{multline*}
Clearly, $H_j(K(\fim)(-1))_d=H_{j-1}(\fim)_{d-1}$ and
$H_0(K(\fim)(-1))_d=0$. The connecting morphism is known to be the
multiplication by the element $\pm f_it$. Thus we get:
\begin{multline*}
  \ldots\to H_1(\fim)_{d-1}\stackrel{\cdot\pm f_it}{\longrightarrow}
  H_1(\fim)_d\to H_1(\fit)_d\to\\
  \to H_0(\fim)_{d-1}\stackrel{\cdot\pm f_it}{\longrightarrow}
  H_0(\fim)_d\to H_0(\fit)_d\to 0.
\end{multline*}
If $d=1$, then $H_1(\fim)_{0}=0$, $H_0(\fim)_{0}=R$ and
$H_0(\fim)_1=I/J_{i-1}$. Hence
\begin{eqnarray*}
\ker\left( H_0(\fim)_{0}\stackrel{\cdot\pm f_it}{\longrightarrow}
H_0(\fim)_1\right)=(J_{i-1}:f_i)=T_{i,1}.
\end{eqnarray*}
In particular, for $i=1,\ldots ,n+1$ and $d=1$, one deduces the exact
sequence:
\begin{eqnarray}\label{eqMain1}
0\to H_1(\fim)_1\to H_1(\fit)_1\to T_{i,1}\to 0,
\end{eqnarray}
where $H_1(\fim)_1=Z_1(f_1,\ldots,f_{i-1})$ and
$H_1(\fit)_1=Z_1(f_1,\ldots,f_i)$.

If $d\geq 2$, one can check that
$H_0(\fim)_{d-1}=I^{d-1}/J_{i-1}I^{d-2}$ and
$H_0(\fim)_{d}=I^{d}/J_{i-1}I^{d-1}$. Thus
\begin{eqnarray*}
\ker\left( H_0(\fim)_{d-1} \stackrel{\cdot\pm f_it}{\longrightarrow}
H_0(\fim)_d\right)=(J_{i-1}I^{d-1}:f_i)\cap
I^{d-1}/J_{i-1}I^{d-2}=T_{i,d}.
\end{eqnarray*}
(See Setting~\ref{setting}.) In particular, for $i=1,\ldots, n+1$ and
$d\geq 2$, we deduce the exact sequence:
\begin{eqnarray}\label{eqMain}
H_1(\fim)_{d-1}\stackrel{\cdot\pm f_it}{\longrightarrow}
H_1(\fim)_d\to H_1(\fit)_d\to T_{i,d}\to 0.
\end{eqnarray}
Note that the middle morphism in \eqref{eqMain} is induced by the
inclusion. Namely, the class of a cycle $(a_1,\ldots,a_{i-1})$,
$a_j\in I^{d-1}$, maps to the class of the cycle
$(a_1,\ldots,a_{i-1},0)$. Similarly, the right-hand morphism is
induced by the projection $(a_1,\ldots,a_i)\mapsto a_i$.
\end{remark}

We recover \cite[Lemma~3.1]{MP12}. Keeping the notations as in
Setting~\ref{setting} and Notation~\ref{notKos}:

\begin{corollary}\label{cortn1}
For $d\geq 2$, the following sequence is exact.
\begin{eqnarray}\label{eqMainn}
0\to \frac{H_1(\fn)_d}{ft\cdot H_1(\fn)_{d-1}}\longrightarrow
E(I)_d\longrightarrow \frac{JI^{d-1}:f^d}{JI^{d-1}:f^{d-1}}\to 0.
\end{eqnarray}
The right-hand morphism sends the class of an equation $P\in Q_d$ to
the class of $P(0,\ldots,0,1)$.
\end{corollary}
\begin{proof}
Take $i=n+1$ and $d\geq 2$ in \eqref{eqMain}. Then
$H_1(\fnp)_d=H_1(\ft)_d$, which by \eqref{effectivekos}, is isomorphic
to $E(I)_d$. See also the definition of $T_{n+1,d}$ and its isomorphic
expression in \eqref{tn1}. The second part follows from the
composition of the morpshims in \eqref{effectivekos}, \eqref{eqMain}
and \eqref{tn1}. Indeed, the class of $P\in Q_d$ is sent to the class
of $(P_1(\f),\ldots,P_{n+1}(\f))\in\bigoplus_{j=1}^{n+1}I^{d-1}$
through \eqref{effectivekos}, where $P=\sum_{j=1}^{n+1}\xi_jP_j$,
$P_j\in A_{d-1}$. By \eqref{eqMain}, $(P_1(\f),\ldots,P_{n+1}(\f))$ is
sent to the class of $P_{n+1}(\f)\in (JI^{d-1}:f)\cap I^{d-1}$.  Write
$P_{n+1}=\sum_{j=1}^{n}\xi_jQ_j+c\xi_{n+1}^{d-1}$, with $Q_j\in
A_{d-2}$ and $c\in R$. In particular, $P_{n+1}(\f)=b+cf^{d-1}$, with
$b=\sum_{j=1}^{n}f_jQ_j(\f)\in JI^{d-2}$. Then the isomorphism
\eqref{tn1} sends the class of $P_{n+1}(\f)$ to the class of $c\in
JI^{d-1}:f^d$. Observe that $P(0,\ldots,0,1)=P_{n+1}(0,\ldots,0,1)=c$.
\end{proof}

The first part of the following result is shown in
\cite[Lema~3.3]{MP12}. Our proof here is a direct consequence of
Remarks~\ref{remshort} and \ref{remlong}, and the sequences
\eqref{eqMain1} and \eqref{eqMain}. We keep the notations as in
Setting~\ref{setting} and Notation~\ref{notKos}.

\begin{theorem}\label{theomany}
Fix $d\geq 1$ and $i=1,\ldots,n+1$.
\begin{itemize}
\item[$(a)$] The following two conditions are equivalent:
\begin{itemize}
\item[$(i)$] $H_1(f_1t)_d=0, H_1(\ftwo)_d=0,\ldots, H_1(\fit)_d=0$;
\item[$(ii)$] $T_{1,d}=0, T_{2,d}=0,\ldots ,T_{i,d}=0$.
\end{itemize}
\item[$(b)$] Suppose that, for some $i=1,\ldots ,n$,
  $T_{1,d}=0,T_{2,d}=0,\ldots,T_{i,d}=0$. Then
\begin{eqnarray*}
  H_1(\fip)_d\cong T_{i+1,d}.
\end{eqnarray*}  
\item[$(c)$] Fix now $d\geq 2$. Suppose that, for some $i=1,\ldots
  ,n-1$,
\begin{eqnarray*}
T_{1,d}=0,T_{2,d}=0,\ldots,T_{i,d}=0\mbox{ and that }
T_{1,d-1}=0,T_{2,d-1}=0,\ldots,T_{i,d-1}=0.
\end{eqnarray*}
Then the following sequence is exact.
\begin{eqnarray}\label{eqi}
  0\to \frac{(J_{i}I^{d-1}:f_{i+1})\cap I^{d-1}}{f_{i+2}\cdot
    [(J_{i}I^{d-2}:f_{i+1})\cap I^{d-2}]+J_{i}I^{d-2}}\to
  H_1(\fipp)_d\to T_{i+2,d}\to 0.
\end{eqnarray}
\end{itemize}
\end{theorem}
\begin{proof} The implication $(i)\Rightarrow (ii)$ follows directly
from the exact sequences \eqref{eqMain1} and \eqref{eqMain}. Since
$H_1(\fzero)=0$, then, by \eqref{eqMain1} and \eqref{eqMain},
$H_1(f_1t)_d\cong T_{1,d}$ and $(ii)\Rightarrow (i)$ holds for
$i=1$. Suppose that $(ii)\Rightarrow (i)$ holds for $i-1\geq 1$.  By
the induction hypothesis, $H_1(\fim)_d=0$ and, by \eqref{eqMain1} and
\eqref{eqMain}, $H_1(\fit)_d\cong T_{i,d}=0$. This proves $(a)$.

Suppose now that, for some $i=1,\ldots ,n-1$, $T_{1,d}=0,
T_{2,d}=0,\ldots,T_{i,d}=0$. In particular, since $(ii)\Rightarrow
(i)$, $H_1(f_1t)_d=0, H_1(\ftwo)_d=0,\ldots, H_1(\fit)_d=0$. Using
\eqref{eqMain1} and \eqref{eqMain}, for the integer $i+1$,
$H_1(\fip)_d\cong T_{i+1,d}$. This proves $(b)$.

Suppose now that the hypotheses in $(c)$ hold. Then, by $(b)$ applied
to $d-1$, we obtain the isomorphism $H_1(\fip)_{d-1}\cong
T_{i+1,d-1}$. Therefore,
\begin{eqnarray*}
H_1(\fip)_{d-1}\cong\frac{(J_iI^{d-2}:f_{i+1})\cap
  I^{d-2}}{J_iI^{d-3}}\mbox{ and } H_1(\fip)_d\cong
\frac{(J_iI^{d-1}:f_{i+1})\cap I^{d-1}}{J_iI^{d-2}}.
\end{eqnarray*}
Through these isomorphisms,
\begin{eqnarray*}
\frac{H_1(\fip)_d}{f_{i+2}t\cdot H_1(\fip)_{d-1}}\cong
\frac{(J_iI^{d-1}:f_{i+1})\cap I^{d-1}}{f_{i+2}\cdot
  [(J_iI^{d-2}:f_{i+1})\cap I^{d-2}]+J_iI^{d-2}}.
\end{eqnarray*}
The rest follows from the exact sequence \eqref{eqMain} applied to
$i+2$.
\end{proof}

\newpage

Next we specialise Theorem~\ref{theomany}~$(c)$, to the case $n=2$. 

\begin{corollary}\label{corkey}
Let $(R,\mfm)$ be a Noetherian local ring, $f_1,f_2,f\in\mfm$ and let
$J= (f_1,f_2)$ and $I=(f_1,f_2,f)$. Fix $d\geq 2$. Assume that
$(0:f_1)\cap I^{d-2}=0$ and $(0:f_1)\cap I^{d-1}=0$. Then the
following sequence is exact.
\begin{eqnarray}\label{eqkey}
0\to \frac{(f_1I^{d-1}:f_{2})\cap I^{d-1}}{f\cdot
  [(f_1I^{d-2}:f_2)\cap I^{d-2}]+f_1I^{d-2}}\longrightarrow
E(I)_d\longrightarrow\frac{(J I^{d-1}: f^d)}{(J I^{d-2}: f^{d-1})}\to 0.
\end{eqnarray}  
\end{corollary}
\begin{proof}
Take $n=2$, $i=1$ and $d\geq 2$ in Theorem~\ref{theomany}~$(c)$. Then
$T_{1,d-1}=(0:f_1)\cap I^{d-2}$ and $T_{1,d}=(0:f_1)\cap I^{d-1}$,
which are zero by hypothesis. The rest follows from the sequence
\eqref{eqi}.
\end{proof}

\begin{corollary}\label{corEquiv}
Let $(R,\mfm)$ be a Noetherian local ring, $f_1,f_2,f\in\mfm$ and let
$J= (f_1,f_2)$ and $I=(f_1,f_2,f)$. Fix $L\geq 2$. Suppose that
$(0:f_1)\cap I^{d-1}=0$, for all $1\leq d\leq L$, and that
$(f_1:f_2)\subseteq (f_1:f)$. Then the following two conditions are
equivalent.
\begin{itemize}
\item[$(a)$] $T_{2,d}=(f_1I^{d-1}:f_2)\cap I^{d-1}/f_1I^{d-2}=0$, for
  all $2\leq d\leq L$;
\item[$(b)$] $E(I)_d\cong (JI^{d-1}:f^2)/(JI^{d-2}:f^{d-1})$, for all
  $2\leq d\leq L$.
\end{itemize}  
\end{corollary}  
\begin{proof}
The hypotheses $(0:f_1)\cap I^{d-1}=0$, for all $1\leq d\leq L$,
allows us to apply Corollary~\ref{corkey}, for all $2\leq d\leq L$.

Suppose that $(a)$ holds. The vanishing of $T_{2,d}$ ensures the
vanishing of the left-hand side term in the exact
sequence~\eqref{eqkey}. Thus $(b)$ holds.

Conversely, assume that $(b)$ holds. Let us prove $T_{2,d}=0$, by
induction on $d$, $2\leq d\leq L$. So take $d=2$. Using $(b)$ and
\eqref{eqkey}, for $d=2$, then $(f_1I:f_2)\cap I=f\cdot
(f_1:f_2)$. Using the hypothesis $(f_1:f_2)\subseteq (f_1:f)$, we get
\begin{eqnarray*}
(f_1I:f_2)\cap I=f\cdot (f_1:f_2)\subseteq f\cdot (f_1:f)\subseteq
  (f_1).
\end{eqnarray*}
Thus $T_{2,2}=0$. Take now $d\geq 3$, $d\leq L$. By the induction
hypothesis $T_{2,d-1}=0$, so
\begin{eqnarray*}
(f_1I^{d-2}:f_2)\cap I^{d-2}=f_1I^{d-3}.
\end{eqnarray*}
Using $(b)$ and \eqref{eqkey}, for such $d$, then
\begin{multline*}
(f_1I^{d-1}:f_{2})\cap I^{d-1}=f\cdot [(f_1I^{d-2}:f_2)\cap
    I^{d-2}]+f_1I^{d-2}=\\f\cdot [f_1I^{d-3}\cap
    I^{d-2}]+f_1I^{d-2}\subseteq f_1I^{d-2}.
\end{multline*}
Hence $T_{2,d}=0$.
\end{proof}

\begin{remark}\label{remregseq}
In the case that $f_1,f_2$ is a regular sequence, then clearly
$(0:f_1)\cap I^{d-1}=0$ for all $1\leq d\leq L$ and $(f_1:f_2)=f_1R
\subseteq (f_1:f)$. However the converse does not always hold as the
next example shows.
\end{remark}

\begin{example}
Let $R=k[[x,y]]$ be the formal power series ring in two variables over
a field $k$ of characteristic zero. Take $a,b\geq 2$ and consider the
ideals $J = (f_1,f_2)$ and $I=(f_1,f_2,f)$ with $f= x^ay^b$, $f_1=
\frac{df}{dx}=ax^{a-1}y^b$ and $f_2= \frac{df}{dy}=b x^{a}y^{b-1}$.
Then $(f_1:f_2)= yR\subseteq (f_1:f) =R,$ whereas $f_1,f_2$ is not a
regular sequence. We point out that $J=I$ is an ideal of linear type.
\end{example}

\section{Ideals with expected relation type}\label{SecDefExp}

We recall now a central concept to our purposes.

\begin{definition}\label{defRed}
Let $(R,\mfm)$ be a Noetherian local ring and let $\mfa$ and $\mfb$ be
two ideals of $R$. The ideal $\mfb$ is a {\em reduction of $\mfa$} if
$\mfb\subseteq \mfa$ and there is an integer $r\geq 0$ such that
$\mfa^{r+1}=\mfb\mfa^r$. From the definition it follows that
$\rad(\mfb)=\rad(\mfa)$, $\Min(R/\mfb)=\Min(R/\mfa)$ and
$\height(\mfb)=\height(\mfa)$ (see, e.g.,\cite[Lemma~8.10]{SH}). Note
that the ideal $\mfa$ is always a reduction of itself.  An ideal
$\mfa$ which has no reduction other than itself is called a {\em basic
  ideal}. The smallest integer $r\geq 0$ satisfying the equality
equality $\mfa^{r+1}=\mfb\mfa^r$ is called the {\em reduction number}
of $\mfa$ with respect to $\mfb$ and is denoted
$\rn_{\mathfrak{b}}(\mfa)$. For $\mfb=\mfa$,
$\rn_{\mathfrak{b}}(\mfa)=0$.  (see \cite{NR54}).
\end{definition}

The next result is shown in \cite[Lemma~3.1]{MP12}. We deduce it here
from our previous remarks.

\begin{proposition}\label{proprn}
Let $(R,\mfm)$ be Noetherian local ring, $n\geq 2$, and let
$J=(f_1,\dots ,f_n)$ be a reduction of $I=(f_1,\dots ,f_n,f)$. Then
\begin{eqnarray*}
\rn_J(I)+ 1 \leq \rt(I).
\end{eqnarray*}  
\end{proposition}
\begin{proof}
Let $\rt(I)=L\geq 1$. Hence, $E(I)_d=0$, for all $d\geq L+1$. By the
exact sequence \eqref{eqMainn}, $T_{n+1,d}=0$, for all $d\geq
n+1$. Therefore $(JI^{d-1}: f^{d})=(JI^{d-2} : f^{d-1})$, for all
$d\geq L+1$. Since $J$ is a reduction of $I$, then $(JI^{m-1}:f^{m}) =
R$, for $m\gg 0$ large enough. Thus $f^L\in JI^{L-1}$,
$I^{L}=JI^{L-1}$ and $\rn_J(I)\leq L-1$.
\end{proof}

\begin{definition}\label{defexp}
Let $(R,\mfm)$ be a Noetherian local ring and let $J=(f_1,\dots ,f_n)$
be a reduction of $I=(f_1,\dots ,f_n,f)$. We say that $I$ has the {\em
  expected relation type with respect to $J$} if
\begin{eqnarray*}
  \rn_J(I)+ 1=\rt(I).
\end{eqnarray*}
When $J$ is understood by the context, we will skip the locution
``with respect to J''.
\end{definition}

\begin{example}\label{exexp}
Let $(R,\mfm,k)$ be a Noetherian local ring with $k=R/\mfm$ an
infinite field. Let $\mfa$ be an ideal of $R$.
\begin{itemize}
\item[$(a)$] If $\mfa$ is of linear type, then $\mfa$ is basic and has
  the expected relation type.
\item[$(b)$] If $\mfa$ is a parametric ideal, that is, generated by a
  system of parameters, then $\mfa$ is basic, but it is not
  necessarily of linear type, nor it has necessarily the expected
  relation type.
\end{itemize}
\end{example}
\begin{proof}
Let $\mfb$ be a reduction of $\mfa$. By \cite[2.~Theorem~1]{NR54},
there exists an ideal $\mfc\subseteq \mfb\subseteq \mfa$, which is a
minimal reduction of $\mfa$, that is, no ideal strictly contained in
$\mfc$ is a reduction of $\mfa$. By \cite[2.~Lemma~3]{NR54}, every
minimal set of generators of $\mfc=(x_1,\ldots,x_s)$ can be extended
to a minimal set of generators of
$\mfa=(x_1,\ldots,x_s,x_{s+1},\ldots,x_{m})$, with
$\mu(\mfc)=s\leq\mu(\mfa)=m$, where $\mu(\cdot)$ stands for the
minimal number of generators. If $\mfa$ is of linear type, then ${\bf
  S}(\mfa)\cong {\bf R}(\mfa)$. On tensoring by $A/\mfm$,
$k[T_1,\ldots,T_m]\cong {\bf F}(\mfa)$, where ${\bf
  F}(\mfa)=\oplus_{d\geq 0}\mfa^d/\mfm \mfa^d$ is the {\em fiber cone
  of} $\mfa$. On taking Krull dimensions, we get
$\mu(\mfa)=m=l(\mfa)$, where $l(\mfa)=\dim {\bf F}(\mfa)$ is the {\em
  analytic spread of} $\mfa$. By \cite[Proposition~4.5.8]{BH93},
$l(\mfa)\leq \mu (\mfc)$. Thus $s=m$ and $\mfc=\mfa$. Therefore
$\mfb=\mfa$ and $\mfa$ is basic. In particular,
$\rn_{\mathfrak{b}}(\mfa)=0$ and, since $\mfa$ is of linear type,
$\rt(\mfa)=1=\rn_{\mathfrak{b}}(\mfa)+1$. This proves $(a)$.

If $\mfa$ is a parametric ideal, then $\height(\mfa)=\mu(\mfa)$. By
\cite[4.~Theorem~5]{NR54}, $\mfa$ is basic. Take now $R=k[[x,y,z,w]]$,
where $w^2=wz=0$, and $\mfa=(x^{m-1}y+z^m,x^m,y^m)$, $m\geq 2$. Then
$\mfa$ is a parameter ideal, hence a basic ideal, but its relation
type is at least $m$ (see \cite[Example~2.1]{AGH}).
\end{proof}

The following result gives a characterization of ideals with expected
relation type in terms of the Koszul homology.

\begin{proposition}\label{propexp}
Let $(R,\mfm)$ be a Noetherian local ring and let $J = (f_1,\dots
,f_n)$ be a reduction of $I=(f_1,\dots ,f_n,f)$. The following
conditions are equivalent.
\begin{itemize}
\item[$(a)$] $I$ has the expected relation type;
\item[$(b)$] $H_1(\fn)_d=ft\cdot H_1(\fn)_{d-1}$, for all $d\geq
  \rn_J(I)+2$.
\end{itemize}
In particular, if $T_{i,d}=(J_{i-1}I^{d-1}:f_i)\cap
I^{d-1}/J_{i-1}I^{d-1}=0$, for all $d\geq \rn_J(I)+2$ and all
$i=1,\ldots ,n$, then $I$ has the expected relation type.
\end{proposition}
\begin{proof}
Set $r=\rn_J(I)$. If $I$ has the expected relation type, then
$E(I)_d=0$, for all $d\geq r+2$. In particular, using the exact
sequence \eqref{eqMainn}, we deduce $H_1(\fn)_d=ft\cdot
H_1(\fn)_{d-1}$, for all $d\geq r+2$. Conversely, if
$H_1(\fn)_d=ft\cdot H_1(\fn)_{d-1}$, for all $d\geq r+2$, then by
\eqref{eqMainn}, $E(I)_d\cong
(JI^{d-1}:f^d)/(JI^{d-1}:f^{d-1})$. However, if $d\geq r+2$, then
$f^{d-1}\in JI^{d-2}$, so $(JI^{d-1}:f^{d-1})=R$ and
$E(I)_d=0$. Therefore, $\rt(I)\leq r+1=\rn_J(I)+1$. The other
inequality follows from Proposition~\ref{proprn}. This shows the
equivalence $(a)\Leftrightarrow (b)$.

If $T_{1,d}=,T_{2,d}=0,\ldots, T_{n,d}=0$, for all $d\geq \rn_J(I)+2$,
by Theorem~\ref{theomany}, $H_1(\fn)_d=0$, for all $d\geq
\rn_J(I)+2$, and, by $(b)\Rightarrow (a)$, $I$ has the expected
relation type. 
\end{proof}

The main result of Mui\~nos and the second author in \cite{MP12},
gives an instance of ideals of expected relation type. We rephrased it
here in terms of our $T_{i,d}$.

\begin{corollary}\label{cormp}
Let $(R,\mfm)$ be a Noetherian local ring, $n\geq 2$. Let
$f_1,\ldots,f_n\in\mfm$ and $f\in\mfm$. Let $J=(f_1,\dots ,f_n)$ and
$I=(f_1,\dots ,f_n,f)$ be ideals of $R$. Assume that for all $d\geq 2$
and all $i=1,\ldots,n$,
\begin{eqnarray*}
T_{i,d}=(J_{i-1}I^{d-1}:f_i)\cap I^{d-1}/J_{i-1} I^{d-2}=0.
\end{eqnarray*}  
Then, for all $d\geq 2$,
\begin{eqnarray*}
E(I)_d\cong \frac{(JI^{d-1}: f^d)}{(JI^{d-2}: f^{d-1})}.
\end{eqnarray*}
In particular, if $J$ is a reduction of $I$, then $I$ has the expected
relation type with respect to $J$. 
\end{corollary}
\begin{proof}
If $T_{1,d}=0,T_{2,d}=0,\ldots,T_{n,d}=0$, for all $d\geq 2$, then, by
Theorem~\ref{theomany}~$(a)$, $H_1(\fn)_d=0$, for all $d\geq
2$. By the exact sequence \eqref{eqMainn}, $E(I)_d\cong (JI^{d-1}:
f^d)/(JI^{d-2}: f^{d-1})$. The second assertion follows directly from
Proposition~\ref{propexp}.
\end{proof}

%\josep[inline]{Nom\'es cal comprovar les condicions fins $r_J(I)+ 1$?}

\section{Divisors of expected Jacobian type}\label{SecDefAlm}

Let $(X,O)$ be a germ of a smooth $n$-dimensional complex variety and
$\cO_{X,O}$ the ring of germs of holomorphic functions in a
neighbourhood of $O$, which we identify with $R=\bC\{x_1,\dots ,
x_n\}$ by taking local coordinates. Let $(D,O)$ be a germ of divisor
defined locally by $f\in R$ and set $f_i=\frac{df}{dx_i}$ for
$i=1\dots, n$. From now on, until the end of the paper, we consider
the following notations.

\begin{setting}\label{settwo}
Let $R=\bC\{x_1,\dots , x_n\}$ be the convergent power series ring,
which is a Noetherian regular local ring (see, e.g.,
\cite[Lemma~7.1]{SH}). Let $f\in \mfm$.  Set $f_i=\frac{df}{dx_i}$,
for $i=1\dots, n$. Let $J=(f_1,\ldots,f_n)$ and
$I=(f_1,\ldots,f_n,f)=(J,f)$.  Note that if $f\in\mfm^2$, then
$f_i\in\mfm$ and $J\subseteq I\subseteq\mfm$.  The ideals $J$ and $I$
are called the {\em gradient ideal of }$f$ and the {\em Jacobian ideal
  of }$f$, respectively. It is known that, when $f\in\mfm$, then $f\in
\overline{(x_1f_1,\ldots,x_nf_n)}\subseteq \overline{J}$, where
$\overline{H}$ stands for the integral closure of the ideal $H$ (see
\cite[Corollary~7.1.4]{SH}). In particular, $J$ is a reduction of $I$
(see, e.g., \cite[Proposition~1.1.7]{SH}).
\end{setting}

\begin{definition}\label{defAlmost}
A germ of divisor $(D,O)$, with reduced equation given by $f$, is of
{\em linear Jacobian type} if $I$ is an ideal of linear type
(\cite[Definition~1.11]{CN09}); $f$ will be said of {\em expected
  Jacobian type}, if $J$ is of linear type and $I$ has the expected
relation type with respect to $J$.
\end{definition}

\begin{remark}\label{remJacImplies}
Divisors of linear Jacobian type are divisors of expected Jacobian
type. Indeed, by Example~\ref{exexp},~$(a)$, if $f$ is of linear
Jacobian type, then $I$ is basic, thus $J=I$ is of linear type and $I$
has the expected relation type since $\rt(I)=1$ and $\rn_{J}(I)=0$.
\end{remark}

\begin{example}\label{exPlaneJaclt}
Let $R=\bC\{x,y\}$ be the convergent power series ring in two
variables $x,y$. Let $f\in\bC[x,y]$ be a polynomial such that
$f\not\in\bC[\lambda x+ y],\bC[x+\mu y]$, for all $\lambda,\mu\in
\bC$. Then $J$ is an ideal of linear type minimally generated by two
elements.
\end{example}  
\begin{proof}
Since $R$ is local and $J=(f_1,f_2)$, to see that $J$ is minimally
generated by two elements it is enough to prove that $f_1\not\in
(f_2)$ or $f_2\not\in (f_1)$. Let us see that if $f_1\in (f_2)$, then
$f$ is either in $\bC[y]$, or else in $\bC[x+\mu y]$, for some
$\mu\in\bC$ (similarly, one would do the same if $f_2\in
(f_1)$). Since $f\not\in\bC[y]$, $\deg_x(f)=r\geq 1$. Write
$f=\sum_{i=0}^rx^rg_i(y)$ and suppose that $f_1=pf_2$, for some
element $p\in\bC\{x,y\}$. Since $f_1$ and $f_2$ are polynomials, then
$p\in\bC[x,y]$ must be a polynomial too. Equating the highest degree
terms in $x$ in the expression $f_1=pf_2$, one deduces that either
$p=0$, or else $g_r^{\prime}(y)=0$.  However, if $p=0$, then $f_1=0$
and $f\in\bC[y]$, a contradiction. Thus $g_r^{\prime}(y)=0$ and
$g_r(y)=a_r\in\bC$, $a_r\neq 0$, because $\deg_x(f)=r\geq
1$. Substituting $g_r(y)=a_r$ in $f$ and equating again the highest
degree term in $x$ in the equality $f_1=pf_2$, one gets
$pg_{r-1}^{\prime}(y)=ra_r\neq 0$. Thus $p\in\bC$, $p\neq 0$. Setting
$\mu=1/p$, we have $g_{r-1}(y)=ra_r\mu y+a_{r-1}$. Again, substituting
this expression in $f$ and equating the $r-2$ degree terms in the
equality $f_1=(1/\mu)f_2$, one gets $g_{r-2}(y)=a_r\binom{r}{2}(\mu
y)^2+a_{r-1}\mu y+a_{r-2}$.  Proceeding recursively, one would get the
equality $f=\sum_{i=0}^{r}a_i(x+\mu y)^i$ and so $f$ would be an
element of $\bC[x+\mu y]$, a contradiction.

Therefore, $J$ is minimally generated by two elements. Now apply
\cite[Proposition~1.5]{Hun}. Thus $J$ can be generated by two elements
which form a $d$-sequence. In particular $J$ is of linear type (see,
e.g., \cite[Corollary~5.5.5]{SH}).
\end{proof}

\begin{example}\label{exIsolated}
Suppose that $J=(f_1,\dots, f_n)$ is generated by an $R$-regular
sequence, for instance, if $f$ has an isolated singularity at $O$
(see, e.g., \cite[IV, Remark~2.5]{Rui93}). In particular, $J$ is of
linear type (\cite[Corollary~5.5.5]{SH}). There are three
possibilities according to the previous definition. If $I$ is of
linear type, then $f$ is a divisor of linear Jacobian type. Suppose
that $I$ is not of linear type. Recall that, by the argument in
Setting~\ref{settwo}, $J$ is a reduction of $I$ and, by
Proposition~\ref{proprn}, $\rn_{J}(I)+1\leq \rt(I)$. If the equality
holds, then $f$ is of expected Jacobian type. The third and last case
occurs when $J$ is of linear type, but $\rn_{J}(I)+1<\rt(I)$, i.e.,
$f$ is not of expected Jacobian type.
\end{example}

\begin{remark}\label{remEuler}
The linear type condition for the Jacobian ideal was investigated by
Calder\-\'on-Mo\-reno and Narv\'aez-Macarro in \cite{CN02,CN09} (see
also \cite{Nar08}). They proved that divisors of linear Jacobian type
are Euler homogeneous. Recall that a divisor $D$ is {\em Euler
  homogeneous} if there is a vector field $\chi$ at $O$ such that
$\chi(f)=f$, or in other words, $f\in J$. In particular, $J=I$ and
$\rn_{J}(I)=0$.
\end{remark}

For divisors satisfying the conditions $T_{i,d}=0$ we can explicitly
describe the equations of the Rees algebra of the Jacobian ideal which
corresponds to describe the blow-up at the singular locus of the
divisor. More precisely, and summarizing several results in a unique
statement:

\begin{theorem}\label{thmEqAlmostT0}
Let $R=\bC\{x_1,\dots , x_n\}$ be the convergent power series ring Let
$f\in\mfm^2$, $J$ and $I$ be as in Setting~\ref{settwo}. Suppose that
$T_{i,d}=0$, for all $d\geq 1$ and all $i=1,\ldots,n$. Set
$\xi_{i+1}=s$. Let
\begin{eqnarray}\label{eqpresjacob}
  \varphi:R[\xi_1,\ldots,x_n,s]=\bC\{x_1,\dots ,
  x_n\}[\xi_1,\ldots,x_n,s]\to {\bf R}(I)
\end{eqnarray}
be a polynomial presentation of ${\bf R}(I)$, sending $\xi_i$ to
$f_it$ and $s$ to $ft$. Let $Q=\bigoplus_{d\geq 1}Q_d$ the ideal of
equations of ${\bf R}(I)$. The following conditions hold.
\begin{itemize}
\item[$(a)$] $f_1,\ldots,f_n$ is an $R$-regular sequence and $J$ is of
  linear type.
\item[$(b)$] $J$ is a reduction of $I$ and $\rt(I)=\rn_J(I)+1$. If
  $f\in J$, then $f$ is a divisor of linear Jacobian type; otherwise,
  $f$ is a divisor of expected Jacobian type.
\item[$(c)$] For all $d\geq 2$, 
\begin{eqnarray}\label{eqEd}
E(I)_d \cong \left(Q/Q\langle d-1\rangle\right)_d\cong
  \frac{(JI^{d-1}:f^{d})}{(JI^{d-2}:f^{d-1})};
\end{eqnarray}
the class of $P(\xi_1,\dots,\xi_n,s)\in Q_d$ is sent to the class of
$P(0,\dots,0,1)\in (JI^{d-1}:f^{d})$.
\item[$(d)$] Set $L=\rt(I)$. A minimal generating set of equations of
  ${\bf R}(I)$ can be obtanied from a minimal generating set of $Q_1$,
  the first syzygies of $I$, and representatives of inverse images of
  a minimal generating set of $(JI^{d-1}:f^{d})/(JI^{d-2}:f^{d-1})$,
  for all $2\leq d\leq L$.
\item[$(e)$] There exists a unique top-degree equation of degree $L$,
  which is of the form
\begin{equation} \label{topL}
s^L+p_1 s^{L-1}+\cdots +p_L,
\end{equation}
where $p_j\in R[\xi_1,\dots ,\xi_n]$ are either zero, or else,
polynomials of degree $j$.
\end{itemize}
\end{theorem}
\begin{proof}
Since $f\in\mfm^2$, then $f_i\in\mfm$ and so $J\subseteq I\subseteq
\mfm$. Recall that $T_{i,1}=(J_{i-1}:f_i)$. Therefore, $T_{i,1}=0$,
for all $i=1,\ldots,n$ is equivalent to $f_1,\ldots,f_n$ being an
$R$-regular sequence. In particular, $I$ is of linear type. This
proves $(a)$; $(b)$ is done in Setting~\ref{settwo}; $(c)$ and $(d)$
are shown in Corollaries~\ref{cormp} and \ref{cortn1}. Finally, since
$L=\rt(I)=\rn_J(I)+1$, then $I^{L}=JI^{L-1}$ and $f^L\in JI^{L-1}$,
which defines an equation of the desired form, namely,
$P=\sum_{j=1}^{n}\xi_jP_j+s^L$, with $P_j \in R[\xi_1,\dots ,
  \xi_n,s]_{L-1}$. The image of the class of this equation through the
isomorphism \eqref{eqEd} is precisely the class of
$P(0,\ldots,0,1)=1$. Note that, for $d=L$, then $(JI^{L-1}:f^L)=R$,
and the isomorphism \eqref{eqEd} is given by $E(I)_L\cong
R/(JI^{d-2}:f^{d-1})$, which says, in particular, that there exists a
unique top-degree equation.
\end{proof}

\begin{remark}
Whenever the conditions $T_{i,d}=0$, for all $d\geq 1$ and all
$i=1,\ldots,n$, do not hold, we need to be more careful when trying to
describe the equations of the Rees algebra of the Jacobian ideal of
$I$. Our guide here will be Corollary~\ref{cortn1}. Note that we may
have non-zero terms on both sides of the short exact sequence
\eqref{eqMainn}. In particular, it may happen that
$\rn_{J}(I)+1<\rt(I)$. However, even in this case, we will have an
equation of the form \ref{topL}, with $L=\rn_{J}(I)+1$. The rest of
the equations of ${\bf R}(I)$ will have degree in $s$ smaller than
$L$, although they may have total degree much bigger.
\end{remark}

%\josep[inline]{Descriure equacions}

\section{An application to {\em D}-module theory}\label{SecConnections}

Let $X$ be a smooth $n$-dimensional complex variety and let $D_X$ be
the sheaf of linear differential operators on $X$ with holomorphic
coefficients. Taking local coordinates at $O\in X$ we will simply
consider $R=\bC\{x_1,\dots , x_n\}$ and its associated ring of
differential operators $D_R=\bC\{x_1,\ldots, x_n\}\langle \partial_1,
\dots,\partial_n\rangle$ where $\partial_i:=\frac{d}{dx_i}$ are the
partial derivatives with respect to the variable $x_i$, $i=1,\ldots
,n$. Notice that $\partial_ix_i - x_i\partial_i=1$ so this is a
non-commutative Noetherian ring whose elements can be expressed in its
normal form as
\begin{eqnarray*}
P:= P({x}, {\partial})= \sum_{\alpha=(\alpha_1,\dots , \alpha_n) \in
  \bZ_{\geq 0}^n} a_\alpha(x_1,\dots, x_n) \partial_1^{\alpha_1}\cdots
\partial_n^{\alpha_n},
\end{eqnarray*}
with finitely many $a_\alpha(x_1,\dots , x_n) \in R$ different from
zero.  The {\em order} of such a differential operator is ${\rm
  ord}(P)= \max \{ |\alpha | \hskip 2mm | \hskip 2mm a_\alpha \neq 0
\}$ and its {\em symbol} is the element in $\bC\{x_1,\dots ,
x_n\}[\xi_1, \dots , \xi_n]$
\begin{eqnarray*}
\sigma(P)= \sum_{|\alpha|={\rm ord}(P)} a_\alpha(x_1,\dots , x_n)
\xi_1^{\alpha_1}\cdots \xi_n^{\alpha_n}.
\end{eqnarray*}
Indeed, we have a filtration $F:=\{F_i\}_{i\in \bZ_{\geq 0}}$ of $D_R$
given by the order, or equivalently setting $\deg (x_i)=0$ and
$\deg(\partial_i)=1$, whose associated graded ring is
\begin{eqnarray*}
gr_F(D_R)\cong\bC\{x_1,\dots , x_n\}[\xi_1,\dots ,\xi_n],
\end{eqnarray*}
with the isomorphism given by sending $\overline{P}\in gr_F(D_R)$ to
the symbol $\sigma(P)$.

More generally we may consider the polynomial ring $D_R[s]$ with
coefficients in $D_R$ whose elements are
\begin{eqnarray*}
  P(s):= P(x,\partial, s)= P_0 s^L + P_{1} s^{L-1} + \cdots + P_L,
\end{eqnarray*}  
with $P_i \in D_R$. The {\em total order} of $P(s)$ is ${\rm
  ord}^T(P(s))= \max \{ {\rm ord}(P_i) + i \hskip 2mm | \hskip 2mm i =
0,\dots, L \}$ and its {\em total symbol} is
\begin{eqnarray*}
\sigma^T(P(s))= \sum_{|\alpha|+ i={\rm ord}^T(P(s))}
a_\alpha(x_1,\dots , x_n) \xi_1^{\alpha_1}\cdots \xi_n^{\alpha_n} s^i
\in \bC\{x_1,\dots , x_n\}[\xi_1, \dots , \xi_n,s].
\end{eqnarray*}
In this case, the filtration $F^T:=\{F_i^T\}_{i\in \bZ_{\geq 0}}$ of
$D_R[s]$ given by the total order, or equivalently setting $\deg
(x_i)=0$, $\deg(\partial_i)=1$ and $\deg(s)=1$, provides an
isomorphism
\begin{eqnarray*}
  gr_{F^T}(D_R[s]) \cong \bC\{x_1,\dots , x_n\}[\xi_1, \dots , \xi_n,s].
\end{eqnarray*}
Given $f\in R$, there exist $P(s)\in D_{R}[s]$ and a nonzero
polynomial $b(s)\in \bQ[s]$ such that
\begin{eqnarray*}
P(s)\cdot f\boldsymbol{f^s}=b(s) \boldsymbol{f^s}.
\end{eqnarray*}
The unique monic polynomial of smallest degree satisfying this
functional equation is called the {\em Bernstein-Sato polynomial} of
$f$. This is an important invariant in the theory of singularities
(see \cite{Gra10}, for further details). The Bernstein-Sato should be
understood as an equation in $R_f[s]\boldsymbol{f^s}$, which is the
free rank-one $R_f[s]$-module generated by the formal symbol
$\boldsymbol{f^s}$. Moreover, $R_f[s]\boldsymbol{f^s}$ has a
$D_R[s]$-module structure given by the action of the partial
derivatives as follows: for $h\in R_f[s]$ we have
\begin{eqnarray*}
\partial_i \cdot h \boldsymbol{f^s} = \left( \frac{d h}{dx_i} +s h
f^{-1}\frac{d f}{dx_i} \right)\boldsymbol{f^s}.
\end{eqnarray*}
Let $D_R[s] \boldsymbol{f^s} \subset R_f[s]\boldsymbol{f^s}$ be the
$D_R[s]$-submodule generated by $\boldsymbol{f^s}$. This module has a
presentation as
\begin{eqnarray*}
D_R[s] \boldsymbol{f^s} \cong \frac{D_R[s]}{{\rm
    Ann}_{D_R[s]}(\boldsymbol{f^s})},
\end{eqnarray*}
where ${\rm Ann}_{D_R[s]}(\boldsymbol{f^s}):= \{ P(s) \in
D_R[s] \hskip 2mm | \hskip 2mm P(s) \cdot \boldsymbol{f^s}=0 \}$.  The
Bernstein-Sato polynomial is the minimal polynomial of the action of
$s$ on
\begin{eqnarray*}
\frac{D_R[s] \boldsymbol{f^s}}{ D_R[s] f\boldsymbol{f^s}}\cong
\frac{D_R[s]}{{\rm Ann}_{D_R[s]}(\boldsymbol{f^s})+(f)}.
\end{eqnarray*}
In order to
study these annihilators we may filter them by the order of the
corresponding differential operators
\begin{eqnarray*}
{\rm Ann}_{D_R[s]}^{(1)}(\boldsymbol{f^s}) \subseteq {\rm
  Ann}_{D_R[s]}^{(2)}(\boldsymbol{f^s}) \subseteq \cdots \subseteq
{\rm Ann}_{D_R[s]}(\boldsymbol{f^s}).
\end{eqnarray*}
A lot of attention has been paid to the case that this chain
stabilizes at the first step.

\begin{definition}
A germ of divisor $(D, O)$, with reduced equation given by $f\in R$,
is of {\em linear differential type} if ${\rm
  Ann}_{D_R[s]}^{(1)}(\boldsymbol{f^s}) = {\rm
  Ann}_{D_R[s]}(\boldsymbol{f^s}),$ that is $ {\rm
  Ann}_{D_R[s]}(\boldsymbol{f^s})$ is generated by total order one
differential operators.
\end{definition}

Divisors of linear differential type have been considered in relation
to several different problems in $D$-module theory as we mentioned in
the Introduction. In \cite[Proposition 3.2]{CN02} the authors proved
that divisors of linear Jacobian type are of linear differential type.

\vskip 2mm

Of course, being a divisor of linear differential type is a very
restrictive condition since, in general, we will have differential
operators of higher total order annihilating $\boldsymbol{f^s}$. Among
these higher order operators there exists a monic one of the form
\begin{eqnarray*}
  P(s)= s^L + P_{1} s^{L-1} + \cdots + P_L,
\end{eqnarray*}
with ${\rm ord}(P_i)\leq i$, which we refer to as the {\em Kashiwara
  operator} (cf. \cite[Theorem 6.3]{Kas77}). This fact prompted Yano
to introduce the following invariants of $f$ (see \cite{Yan78},
\cite{Yan83}).
\begin{notation}\label{notyano}
The {\em Kashiwara number} of $f$ is
\begin{eqnarray*}
L(f):=\min\{L\mid P(s)=s^L+P_{1}s^{L-1}+\cdots +P_L\in {\rm
  Ann}_{D_R[s]}(\boldsymbol{f^s})\mbox{ , }{\rm ord}(P_i)\leq i\}.
\end{eqnarray*}
Moreover, set
\begin{itemize}
\item $r(f):=\min\{\,r\geq 1\mid f^r\in J\}$,
\item $id(f):=\min\{\,r\geq 1\mid f^r\in JI^{r-1}\}=\mbox{ integral
  dependence of }f$.
\end{itemize}
Clearly, $id(f)=\rn_J(I)+1$. Yano proved that
\begin{eqnarray}\label{eqYano}
r(f) \leq id(f) \leq L(f).
\end{eqnarray}
Furthermore, if $f$ is quasi-homogeneous we have that $L(f)=1$.
\end{notation}

Yano \cite{Yan78} was able to compute the Bernstein-Sato polynomial of
$f$, when $L(f)=2,3$ by giving an explicit free resolution of ${\rm
  Ann}_{D_R[s]}(\boldsymbol{f^s})$. This invariant also plays a
prominent role in the algorithm presented in \cite{BGMM} to compute
Bernstein-Sato polynomials of isolated singularities that are nondegenerate with respect to its Newton polygon.

\vskip 2mm

Notice that the symbol of the Kashiwara operator resembles the
equation of the Rees algebra of the Jacobian ideal of top degree in
$s$. So we would like to get deeper insight into this relation. First,
since $gr_{F^T}(D_R[s])\cong \bC\{x_1,\dots ,x_n\}[\xi_1,\dots
  ,\xi_n,s]$, we may interpret the presentation of the Rees algebra of
the Jacobian ideal given in Equation~\ref{eqpresjacob} as a surjective
morphism
\begin{eqnarray*}
\varphi: gr_{F^T}(D_R[s])\hskip 2mm \lra \hskip 2mm {\bf R}(I).
\end{eqnarray*}
%that sends $\xi_i$ to $f_i t$ and $s$ to $f t$.  
A key result that can be found in \cite[\S I]{Yan78} (see \cite[Lemma
  1.9]{CN09}, for more details) states that
\begin{eqnarray*}
\sigma^T( {\rm Ann}_{D_R[s]}(\boldsymbol{f^s})) \subseteq \ker\varphi,
\end{eqnarray*}
where $\sigma^T( {\rm Ann}_{D_R[s]}(\boldsymbol{f^s})) = \langle
\sigma^T(P(s) ) \hskip 2mm | \hskip 2mm P(s) \in {\rm
  Ann}_{D_R[s]}(\boldsymbol{f^s}) \rangle$. Indeed, it follows from
\cite[\S5]{Kas77} and \cite[Proposition 2.3]{Yan78} that
\begin{eqnarray*}
\rad(\sigma^T( {\rm Ann}_{D_R[s]}(\boldsymbol{f^s}))) = \ker\varphi.
\end{eqnarray*}
Therefore there exists some non-negative integer $\ell \in \bZ_{\geq
  0}$ such that $(\ker \varphi)^\ell \subseteq \sigma^T({\rm
  Ann}_{D_R[s]}(\boldsymbol{f^s}))$. We also point out that $\ker
\varphi$ is a prime ideal since ${\bf R}(I) \subset R[t]$ is a domain.

\begin{proposition}\label{Lf}
Let $(D,O)$ be a germ of a divisor of expected Jacobian type, with
reduced equation given by $f\in R$. Let $\ell \in \bZ_{\geq 0}$ be the
smallest non-negative integer such that $(\ker \varphi)^\ell$ is
contained in $\sigma^T({\rm Ann}_{D_R[s]}(\boldsymbol{f^s}))$. Then
\begin{eqnarray}\label{eqKashi}
  \rt(I)\leq L(f)\leq\ell\cdot\rt(I).
\end{eqnarray}
In particular, if $\sigma^T({\rm
  Ann}_{D_R[s]}(\boldsymbol{f^s}))=\ker\varphi$, then $\rt(I)=L(f)$.
\end{proposition}
\begin{proof}
It is readily seen that $id(f)=\rn_J(I)+1$. Moreover, since $f$ is of
expected Jacobian type, then $\rt(I)=\rn_J(I)+1$, and so,
$\rt(I)=id(f)\leq L(f)$ (see \eqref{eqYano}). Furthermore, the ideal
of equations of the Rees algebra of $I$ contains an element of the
form
\begin{eqnarray*}
s^{r+1}+p_1s^{r}+\cdots +p_{r+1}\in\ker\varphi,
\end{eqnarray*}
where $r+1=\rt(I)=id(f)$, and where each $p_j$ is either zero, or else
a polynomial of degree $j$. Therefore, we have that
\begin{eqnarray*}
(s^{r+1}+p_1s^{r}+\cdots +p_{r+1})^\ell\in(\ker\varphi)^\ell\subseteq
  \sigma^T({\rm Ann}_{D_R[s]}(\boldsymbol{f^s})),
\end{eqnarray*}  
so there exists a Kashiwara operator of degree at most
$\ell\cdot\rt(I)$.
\end{proof}

The upper bound in \eqref{eqKashi} is far from being sharp as we will
see in the examples of Section~\ref{SecExamples}. The issue here is
how to lift an equation of the Rees algebra to a differential operator
that annihilates $\boldsymbol{f^s}$. We would like to mention that
necessary conditions for the existence of such a lifting were already
given in \cite{Yan78}.

%\josep[inline]{Yano dona condicions suficients per veure si $p\in
%\ker \varphi_f$ \'es tal que $\sigma^T(P(s))=p$}

\section{Examples}\label{SecExamples}

Let $R=\bC\{x_1,\dots , x_n\}$ be the ring of convergent series with
coefficients in $\bC$ and, for a given $f\in\mfm$, let $I=(f_1,\ldots,
f_n,f)$ and $J=(f_1,\ldots, f_n)$ be the Jacobian and gradient ideal
of $f$. We know that $J$ is a reduction of $I$. If moreover $J$ is
generated by a regular sequence, then $J$ is of linear type (see
Setting~\ref{settwo} and Example~\ref{exIsolated}).

\vskip 2mm

The aim of this section is to illustrate with some examples the
condition of having the expected relation type and compare the
relation type of $I$ with the invariant $L(f)$. In order to do so we
will use the mathematical software packages {\tt Macaulay2}
\cite{GS}, {\tt Magma} \cite{Magma}, and {\tt Singular} \cite{DGPS},
to test the sufficient conditions
\begin{eqnarray*}
  T_{i,d}=(J_{i-1}I^{d-1}:f_i)\cap I^{d-1}/J_{i-1} I^{d-2}=0
\end{eqnarray*}  
considered in Theorem~\ref{thmEqAlmostT0}. Of course, $T_{1,d}=0$ is
always satisfied since $f_1$ is a nonzero divisor.  In the case of plane curves  we will only have to check whether the  conditions
$T_{2,d}=0$ hold for $d\geq 2$.

\vskip 2mm

\noindent {\bf Warning}: Notice that we have to deal with an
infinite set of conditions, namely, the vanishing of the modules
$T_{i,d}$. Up to now, we do not know how to solve this difficulty. 
In the examples we present we could compute $T_{i,d}$ for all values of $d$ up to a positive integer much
larger than the reduction number of $I$. This suggests that the examples we deal with 
are of the expected Jacobian type but, in no case, our computations should be considered as formal
proofs.  We should also point out that the calculations of these colon ideals tend to be
extremely costly from the computational point of view.

%According to
%Corollary~\ref{corEquiv} and Remark~\ref{remregseq}, if $f$ is a plane
%curve with $f_1,f_2$ an $R$-regular sequence, the conditions
%$T_{2,d}=0$ hold for all $d\geq 2$ if, and only if, for every $d\geq
%2$, the module of effective relations is
%\begin{eqnarray*}
%E(I)_d \cong \frac{(J I^{d-1}: f^d)}{(JI^{d-2}:f^{d-1})}.
%\end{eqnarray*}

\vskip 2mm

When computing the invariant $L(f)$ we will use the {\tt Kashiwara.m2}
package \cite{BL}.

\subsection{Some examples of plane curves}

Let $f\in R=\bC\{x_1,x_2\}$ be the equation of a germ of plane curve.
It is proved in \cite[Proposition 2.3.1]{Nar08} that $f$ is a divisor
of linear Jacobian type if and only if $f$ is quasi-homogeneous, so
this is a very restrictive assumption.  Indeed, for irreducible plane
curves, this corresponds to the case where $f=x^a+y^b$, with
$\gcd(a,b)=1$, a particular case of irreducible curve with one
characteristic exponent. It is well known that the Bernstein-Sato
polynomial varies within a deformation with constant Milnor number. We
are going to test the behaviour of the expected relation type property
with some examples  of irreducible plane curves with an isolated singularity at the origin.

\vskip 2mm 

$\bullet$ {\bf Reiffen curves:} We consider $f=x^a+y^b+xy^{b-1}$ with
$b\geq a+1$, $a\geq 4$. This family of irreducible plane curves with one characteristic exponent has been a recurrent example in the
theory of $D$-modules. In the case  $a=4$, Nakamura \cite{Nak08, Nak16}  gave a a description of ${\rm
  Ann}_{D_R[s]}(\boldsymbol{f^s})$ and showed that $L(f)=2$.

\vskip 2mm 

%One can computationally check that $T_{2,1}=(J_1:f_2)=(f_1:f_2)=0$,
%that is, $f_1,f_2$ is a regular sequence and, in particular, $J$ is an
%ideal of linear type.
We have tested many examples of Reiffen curves varying the values of $a$ and $b$. In all the cases we checked
that $T_{2,d}=0$, for all
$d\geq 2$.  Thus, by e.g.,
Corollary~\ref{cormp}, $I$ has the expected relation type with respect
to $J$, in particular, $f$ is a divisor of expected Jacobian type (see
Definition~\ref{defAlmost}). By Corollary~\ref{corEquiv} and Remark~\ref{remregseq}, we have 
$E(I)_d\cong (J I^{d-1}: f^d)/(JI^{d-2}:f^{d-1})$ and, further computations suggest that 
$(JI^{d-1}:f^d)=((b-1)x+by,y^{a-1-2d})$, for all $d \leq \lfloor a/2 \rfloor -1$, and
$(JI^{d-1}:f^d)=R$, for $d\geq \lfloor a/2 \rfloor$. In particular,
$\rt(I)=\rn_J(I)+1=\lfloor a/2 \rfloor $.

%We have tested many examples of Reiffen curves
%and all of them satisfy $T_{2,d}=0$ for many values of $d$. The
%relation type that we obtained is always $\rt(I)=a-2$.

\vskip 2mm 

$\bullet$ {\bf Irreducible curves with one characteristic exponent:}
More generally we consider deformations with constant Milnor number of
irreducible curves with one characteristic exponent which have the
form
$$f=x^a+y^b - \sum t_{i,j} \hskip 1mm x^i y^j,$$ where gcd$(a,b)=1$ and the sum is
taken over the monomials $x^i y^j$ such that $0\leq i \leq a-2$,
$0\leq j \leq b-2$ and $bi + aj > ab$. It is well known that these
curves belong to the same equisingularity class but their analytic
type varies depending of the parameters $t_{i,j}$. In particular, the
Bernstein-Sato polynomial also varies and there exists a
stratification of the space of parameters with a specific
Bernstein-Sato polynomial at each strata (see \cite{Kat81, Kat82} and
\cite{CN}).

\vskip 2mm 

%\josep[inline]{Tots els casos que he provat semblen ser almost
%linear Jacobian type}

We take for example the following case considered by Kato
\cite{Kat81}

\vskip 2mm 

$\cdot$ Let $f=x^7+y^5 - t_{3,3} x^3y^3 - t_{5,2} x^5y^2 - t_{4,3}
x^4y^3 - t_{5,3} x^5y^3$. The stratification given by the
Bernstein-Sato polynomial with its corresponding $L(f)$ invariant is:

\vskip 2mm 

\hskip 1cm $\{ t_{5,2} \neq 0 , 6t_{5,2} +175t_{3,3}^4=0 \}$. We have
$L(f)=2$.

\hskip 1cm $\{ t_{5,2} \neq 0 , 6t_{5,2} +175t_{3,3}^4\neq 0 \}$. We
have $L(f)=3$.

\hskip 1cm $\{ t_{3,3} = 0 , t_{5,2}t_{4,3}\neq 0 \}$. We have
$L(f)=3$.

\hskip 1cm $\{ t_{3,3} = 0 , t_{5,2}\neq 0, t_{4,3}= 0 \}$. We have
$L(f)=2$.

\hskip 1cm $\{ t_{3,3} = 0 , t_{5,2}=0, t_{4,3}\neq 0 \}$. We have
$L(f)=2$.

\hskip 1cm $\{ t_{3,3} = 0 , t_{5,2}=0 , t_{4,3}= 0, t_{5,3}\neq 0
\}$. We have $L(f)=2$.

\hskip 1cm $\{ t_{3,3} = 0 , t_{5,2}=0 , t_{4,3}= 0, t_{5,3}= 0
\}$. We have $L(f)=1$.

\vskip 2mm 

For any representative in each strata that we considered we checked
out that  $T_{2,d}=0$  for all the
values of $d$ that we could compute. Moreover, the relation type is
always $\rt(I)=2$ except for the last case which obviously
corresponds to the homogeneous case. In particular, there are strata
in which we have an strict inequality $\rt(I) < L(f)$.

\vskip 2mm 

We have tested several other examples of irreducible plane curves with
one characteristic exponent and all of them satisfied the conditions
$T_{2,d}=0$. This suggests that this class of plane curves have
the expected relation type and we can describe the module of effective
relations using Corollary~\ref{cormp}.

\vskip 2mm

$\bullet$ {\bf Irreducible curves with two characteristic exponents:}
We start considering the simplest example of such a plane curve which
is:

\vskip 2mm 

$\cdot$ Let $f=(y^2-x^3)^2 -x^5y$. We have $L(f)=2$.  We checked out
that condition $T_{2,d}=0$ is satisfied for all the values of
$d$ that we could compute and that the relation type is $\rt(I)=2$ so
it seems to have the expected relation type.

\vskip 2mm

However we can find examples with two characteristic exponents not
satisfying  $T_{2,d}=0$. For example,

\vskip 2mm

$\cdot$ Let $f=(y^2-x^3)^5 -x^3y^{10}$. The reduction number of $I$ with respect to $J$ is $\rn_J(I)=4$. We checked that $T_{2,2}=0$, $T_{2,3} \neq 0$, $T_{2,4}\neq 0$ and $T_{2,5}\neq 0$  but $T_{2,d}=0$  for all the values $d\geq 6$ that we could compute. According to Corollary~\ref{corkey} we have
\begin{eqnarray*}
0\to \frac{(f_1I^{d-1}:f_{2})\cap I^{d-1}}{f\cdot
  [(f_1I^{d-2}:f_2)\cap I^{d-2}]+f_1I^{d-2}}\longrightarrow
E(I)_d\longrightarrow\frac{(J I^{d-1}: f^d)}{(J I^{d-2}: f^{d-1})}\to 0.
\end{eqnarray*} 
and the right term is zero for $d\geq 6$. Even though the conditions $T_{2,d}=0$ are not satisfied for all $d\geq 2$ we have that the left term of the short exact sequence is zero  for $d\geq 6$. Therefore the relation type is $\rt(I)=5$ so
$f$ has the expected relation type. Notice that the modules of effective relations are not
as easy to describe as in Corollary \ref{corEquiv}.

%as in the case of Corollary~\ref{corkey}. In this
%case we would have for $d=3,4,5$

%\begin{eqnarray*}
%0 \hskip 2mm \lra \hskip 2mm \frac{H_1(f_1t, f_{2} t; {\bf R}(I)
%  )_{d}}{ft\cdot H_1(f_1t,f_{2}t; {\bf R}(I))_{d-1}} \hskip 2mm
%\lra \hskip 2mm E(I)_{d} \hskip 2mm \longrightarrow \hskip 2mm
%\frac{(JI^{d-1}:f^{d})}{(JI^{d-2}:f^{d-1})} \hskip 2mm \lra \hskip 2mm
%0
%\end{eqnarray*}
%and for $d\neq 3,4$,
%\begin{eqnarray*}
%E(I)_{d} \cong\frac{(JI^{d-1}:f^{d})}{(JI^{d-2}:f^{d-1})}.
%\end{eqnarray*}

\vskip 2mm

Our computer runs out
of memory before computing the invariant $L(f)$.

\vskip 2mm 

\subsection{Some examples which have not the expected relation type}

Narv\'aez-Macarro \cite{Nar08} considered some examples of
non-isolated singularities which are not of linear Jacobian type. We
will revisit them from our own perspective. We point out that these
examples satisfy $J=I$ so the effective relations for $d\geq 2$ are
characterized by Corollary~\ref{cortn1}. All these examples satisfy
$L(f)=1$ but the relation type is strictly bigger than one.

\vskip 2mm 

$\cdot$ $f=xy(x+y)(x+yz)$. 

\vskip 2mm 

We have that $T_{2,d}=0$  for all the values of $d$ that we could compute.
On the other hand, $T_{3,2}\neq 0$  but $T_{3,d}=0$  for
all the values of $d\geq 3$ that we could compute.  Notice that we are in the situation where $$E(I)_d \cong
\frac{H_1(f_1t, f_{2} t, f_{3} t; {\bf R}(I) )_{d}}{f t H_1(f_1t,
  f_{2} t, f_{3} t; {\bf R}(I) )_{d-1}},$$ but in this case $E(I)_d=0$ for all
$d\geq 3$, i.e. $\rt(I)=2$, as it was described in \cite{Nar08}.

%According to
%Theorem~\ref{theomany},~$(b)$, this suggests that $E(I)_d=0$ for all
%$d\geq 3$ and thus $\rt(I)=2$.

\vskip 2mm

$\cdot$ $f=(xz+y)(x^k-y^k)$. 

\vskip 2mm 

We have that $T_{2,d}=0$  for all the values of $d$ that we could
compute. However:

\vskip 2mm

\hskip 1cm $ \cdot \hskip 2mm k=4: $ $T_{3,2} \neq 0$, but
$T_{3,d}=0$  for $d\geq 3$. This suggests that $\rt(I)=2$.

\hskip 1cm $ \cdot \hskip 2mm k=7: $ $T_{3,2} \neq 0$, $T_{3,3}\neq 0$ and
$T_{3,4} \neq 0$  but $T_{3,d}=0$  for $d\geq 5$. Thus
$\rt(I)=4$.

\vskip 2mm Here we also have $$E(I)_d \cong
\frac{H_1(f_1t, f_{2} t, f_{3} t; {\bf R}(I) )_{d}}{f t H_1(f_1t,
  f_{2} t, f_{3} t; {\bf R}(I) )_{d-1}}.$$ For $k=4$ we have
$H_1(f_1t, f_{2} t, f_{3} t; {\bf R}(I) )_{d}=0$ if $d\geq 3$ and, in
the case that $k=7$, we have $H_1(f_1t, f_{2} t, f_{3} t; {\bf R}(I)
)_{d}=0$ if $d\geq 5$. This also follows from the computations done in \cite{Nar08}.

%\vskip 5mm

%{ {\bf Acknowledgements:} This work grew up from early conversations with Ferran Mui\~nos and we are really grateful for his insight.
%We would also thank Jos\'e Mar\'ia Giral for some helpful comments. Part of this work was done during a research stay of the first author at CIMAT, Guanajuato with a Salvador de Maradiaga grant (ref. PRX 19/00405) from the Ministerio de Ciencia, Innovaci\'on y Universidades.} 

\end{document}